\documentclass[12pt]{amsart}
\usepackage[all]{xy}
\usepackage[mathscr]{euscript}
\usepackage{hyperref}
\hypersetup{colorlinks=true,urlcolor=blue,citecolor=blue,linkcolor=blue}
\usepackage{courier}
\usepackage{amsmath,amscd,amssymb,amsthm,latexsym,longtable,diagram, picture}
\usepackage{graphicx}
\usepackage{array}
\usepackage{color}
\usepackage{enumerate}
\usepackage{nicefrac}
\usepackage{listings}
\usepackage{latexsym,bm,bbm,mathrsfs}
\usepackage{hyperref,graphicx, enumerate}
\usepackage{epsfig, xcolor, graphicx}
\usepackage[shortlabels]{enumitem}
\usepackage{indentfirst, setspace}
\usepackage{tikz-cd}

\allowdisplaybreaks
\usepackage{caption}
\usepackage{cancel}
\usepackage[normalem]{ulem}

\usetikzlibrary{calc,matrix,arrows,decorations.markings}

\newcommand{\bbz}{\mathbb{Z}}
\newcommand{\bbq}{\mathbb{Q}}
\newcommand{\bbp}{\mathbb{P}}
\newcommand{\bbc}{\mathbb{C}}

\newcommand{\bba}{\mathbb{A}}
\newcommand{\bbf}{\mathbb{F}}

\newcommand{\End}{\mathrm{End}}
\newcommand{\Aut}{\mathrm{Aut}}

\newcommand{\mat}{\begin{pmatrix}}
\newcommand{\emat}{\end{pmatrix}}

\newtheorem{theorem}{Theorem}[section]
\newtheorem{proposition}[theorem]{Proposition}
\newtheorem{corollary}[theorem]{Corollary}
\newtheorem{lemma}[theorem]{Lemma}
\newtheorem{rmk}[theorem]{Remark}

\begin{document}
 \title[Additive Rigidity for Images of Rational Points on Abelian Varieties]
{Additive Rigidity for Images of Rational Points on Abelian Varieties I: The simple case}

\author{Seokhyun Choi}

\address{
Dept. of Mathematical Sciences, KAIST,
291 Daehak-ro, Yuseong-gu,
Daejeon 34141, South Korea
}
\email{sh021217@kaist.ac.kr}

\date{\today}

\begin{abstract}
    We study the interaction between the group law on an abelian variety and the additive structure induced on its image under a morphism to a projective space. Let $A/F$ be a simple abelian variety, $f:A \rightarrow \mathbb{P}^n$ be a morphism which is finite onto its image, and $\Gamma \subseteq A(F)$ be a finite-rank subgroup. We show that for any affine chart $\mathbb{A}^n \subseteq \mathbb{P}^n$ and any finite subset $X \subseteq f(\Gamma) \cap \mathbb{A}^n$, the energy satisfies $E(X) \ll \lvert X \rvert^2$ and the sumset satisfies $\lvert X+X \rvert \gg \lvert X \rvert^2$. We also prove a product version of the main theorem, where the morphism is compatible with the decomposition of the abelian variety into simple factors. The proof uses the uniform Mordell-Lang conjecture proven by Gao--Ge--K\"{u}hne.
\end{abstract}

\maketitle

\section{Introduction}\label{Introduction}

Many problems in number theory arise from the interaction of distinct algebraic structures. A fundamental example is provided by the Mordell-Weil group of an abelian variety defined over a number field. On the one hand, an abelian variety carries its intrinsic group law. On the other hand, after applying a morphism to a projective space and restricting to an affine chart, one obtains additive subsets of an affine space. It is then natural to ask how these two additive structures interact.

The central theme of this paper is that the group law on an abelian variety and the addition in affine space do not interact freely. When one passes to affine space via a projective morphism, this interaction is constrained and leads to strong restrictions on the additive structure of the image. From this point of view, our main result can be interpreted as an additive rigidity statement: the image of finite rank subgroups of simple abelian varieties cannot have strong additive structures, with bounds depending only on the geometric data and the rank. This phenomenon already appears in the case of elliptic curves, which we now recall.

The primary motivation for this work comes from the author's earlier paper \cite{Cho25}, which studies the case of $x$-coordinates of rational points on elliptic curves. In that paper, one proves that if the set of $x$-coordinates of rational points on an elliptic curve lies inside a generalized arithmetic progression with positive density, then the number of such points is bounded exponentially in the Mordell-Weil rank. In particular, \cite[Corollary~1.3]{Cho25} shows that if a finite set of $x$-coordinates has a small sumset, then its cardinality must be bounded exponentially in the Mordell-Weil rank. The present paper is motivated by the expectation that the same rigidity phenomenon
persists under the following three generalizations:
\begin{enumerate}
    \item replacing the elliptic curve $E$ by a simple abelian variety $A$;
    \item replacing the $x$-coordinate morphism by a morphism $f:A \rightarrow \bbp^n$ which is finite onto its image;
    \item replacing the field $\bbq$ by arbitrary number field $K$.
\end{enumerate}
In fact, we work more generally with an arbitrary finite-rank subgroup
\[\Gamma \subseteq A(\overline{K}),\]
rather than restricting to $K$-rational points. The advantage of this final generalization is that it allows us to treat torsion points of the abelian variety within the same framework.

Another major source of motivation is the recent work of Harrison, Mudgal, and Schmidt \cite{HMS26} on the sum-product phenomenon for algebraic groups. Their \cite[Theorem~2.1]{HMS26} shows that if $G$ and $H$ are algebraic groups of dimension $1$, if $\mathcal C_1,\ldots,\mathcal C_g$ are correspondences of degree $d$ between $G$ and $H$, and if $A$ is contained in a finite-rank subgroup of $G(\bbc)$, then under suitable nondegeneracy conditions, 
\[\lvert \mathcal C_1(A)+\cdots+\mathcal C_g(A) \rvert \gg \lvert A \rvert^g.\]
This strongly suggests that analogous sumset phenomena should persist in higher-dimensional algebraic groups. The present paper constitutes a first step in this direction, as our theorem applies to simple abelian varieties of arbitrary dimension. However, our method is subject to two main limitations. First, we treat only the two-fold sumset $X+X$, rather than general $m$-fold sumsets. Second, we restrict to correspondences arising from a morphism 
\[f:A \rightarrow \bbp^n\]
which is finite onto its image, and do not consider general correspondences between general algebraic groups. These restrictions arise from the nature of our method, which is based on a detailed analysis of the fibers of the equation
\[f(P)+f(Q)=u,\quad u \in \bba^n\]
inside $A^2$.

In a related direction, Shkredov \cite{Shk23} studied the multiplicative energy of subsets of algebraic varieties over finite fields. While his work concerns multiplicative energy over finite fields, the present paper studies additive energy of abelian varieties over fields of characteristic zero.

Finally, in another related direction, Caro and Garcia-Fritz \cite{CG25} studied linear relations among the $x$-coordinates of triples of points on elliptic curves. Using the uniform Mordell--Lang conjecture, they proved uniform bounds for the number of such relations. Their proof also uses an analysis of exceptional geometric configurations together with the uniform Mordell--Lang conjecture.

Our main Diophantine tool is the uniform Mordell-Lang conjecture proved by Gao, Ge, and K\"{u}hne \cite{GGK21}. Historically, this theorem builds on a line of developments beginning with the proof of Mordell-Lang conjecture by Faltings \cite{Fal91,Fal94} and continuing through uniform and quantitative refinements due to R\'{e}mond \cite{Rem00}, David--Philippon \cite{DP07}, Dimitrov--Gao--Habegger \cite{DGH21}, and Gao--Ge--K\"{u}hne \cite{GGK21}. The constant appearing in the uniform Mordell-Lang conjecture of Gao--Ge--K\"{u}hne \cite{GGK21} depends only on the dimension of an abelian variety, the degree of a subvariety, and the rank of a finite-rank subgroup. 

For a finite subset $X$ of a torsion-free abelian group, the sumset $X+X$ is defined by 
\[X+X := \{a+b \:|\: a,b \in X\},\]
and the energy of $X$ is defined by 
\[E(X) := \lvert \{(a,b,c,d) \in X^4\:|\:a+b=c+d\} \rvert.\]

Our main result is the following.

\begin{theorem}\label{main_theorem}
    Let $A/F$ be a simple abelian variety of dimension $g$ over an algebraically closed field $F$ of characteristic 0. Let $f:A \rightarrow \bbp^n$ be a morphism which is finite of degree $d$ onto its image, and let $t$ denote the projective degree of $f(A)$ in $\bbp^n$. Let $\Gamma$ be a subgroup of $A(F)$ of finite rank $r$. Then there exists a constant $C(g,d,t)>0$ with the following property.
    
    For every affine chart $\bba^n \subseteq \bbp^n$ and every finite subset $X \subseteq f(\Gamma) \cap \bba^n$, 
    \[E(X) \leq C(g,d,t)^{1+r}\lvert X \rvert^2,\qquad \lvert X+X \rvert \geq \left(C(g,d,t)^{-1}\right)^{1+r}\lvert X \rvert^2.\]
\end{theorem}

\begin{rmk}
    By the Cauchy-Schwarz inequality \eqref{Cauchy_Schwarz_X}, 
    \[\lvert X \rvert^4 \leq E(X) \lvert X+X \rvert.\]
    Therefore, it suffices to establish the bound 
    \[E(X) \leq C(g,d,t)^{1+r}\lvert X \rvert^2\]
    in Theorem~\ref{main_theorem}.
\end{rmk}

The simplicity assumption in Theorem~\ref{main_theorem} is removed in the sequel \cite{Cho26}, where the same additive rigidity statement is proved for arbitrary abelian varieties. The present simple-case argument nevertheless remains an important ingredient in the proof of the general case. In particular, Theorem~\ref{main_theorem} provides the base case for the induction.

In additive combinatorics, finite sets with small sumset or large energy are known to be highly additively structured. For instance, Freiman-type theorems show that if a finite subset $X$ satisfies $\lvert X+X \rvert \ll \lvert X \rvert$, then $X$ must be contained in a generalized arithmetic progression $G$ with positive density. From this point of view, Theorem~\ref{main_theorem} shows that images of finite rank subgroups of simple abelian varieties cannot have additive structures, up to constants depending only on the geometric data and the rank.

When $A$ is an elliptic curve $E$, and $f:A \rightarrow \bbp^n$ is the $x$-coordinate morphism $x:E \rightarrow \bbp^1$, one has $g=1$, $d=2$, and $t=1$. Hence Theorem~\ref{main_theorem} gives the following corollary.

\begin{corollary}\label{elliptic_curve_corollary}
    Let $E/\bbq$ be an elliptic curve of Mordell-Weil rank $r$. Then there exists an absolute constant $C>0$ with the following property.

    For every finite subset $X \subseteq x(E(\bbq)) \cap \bba^1$, 
    \[E(X) \leq C^{1+r} \lvert X \rvert^2,\qquad \lvert X+X \rvert \geq \left(C^{-1}\right)^{1+r}\lvert X \rvert^2\]
    
    In particular, if 
    \[\lvert X+X \rvert \leq K\lvert X \rvert \quad\text{or}\quad E(X) \geq \frac{\lvert X \rvert^3}{K}\]
    for some constant $K>0$, then
    \[\lvert X \rvert \leq K C^{1+r}.\]
\end{corollary}

Hence, the small sumset \cite[Corollary~8.3]{Cho25} and large energy \cite[Corollary~8.7]{Cho25} consequences established in \cite[Section~8]{Cho25} follow immediately from Theorem~\ref{main_theorem}. Moreover, our result removes the dependence on the elliptic curve $E$ in the constant, without assuming Lang's conjecture. 

However, there is an effectiveness issue. The absolute constant $C$ appearing in Corollary~\ref{elliptic_curve_corollary} is ineffective, since it depends on the ineffective constant $c(g,d)$ in the Mordell-Lang conjecture of \cite{GGK21}. By contrast, by \cite{Cho25}, one can obtain an effective bound $C$ whenever the constant $c_L$ appearing in Lang's conjecture \cite[Conjecture~1.4]{Cho25} is effective. This is indeed the case for the family of elliptic curves considered in \cite{Cho25}.

Theorem~\ref{main_theorem} also yields a particularly clean consequence for torsion points, since $\Gamma$ has rank 0 in this case.

\begin{corollary}\label{torsion_corollary}
    Let $A/F$ and $f:A \rightarrow \bbp^n$ be as in Theorem~\ref{main_theorem}. Then for every affine chart $\bba^n \subseteq \bbp^n$ and every finite subset $X \subseteq f(A(F)_{\mathrm{tors}})\cap \bba^n$, 
    \[E(X) \leq C(g,d,t) \lvert X \rvert^2,\qquad \lvert X+X \rvert \geq C(g,d,t)^{-1}\lvert X \rvert^2.\]
\end{corollary}

We also prove a product version of Theorem~\ref{main_theorem}. Suppose that an abelian variety is given as a product of simple factors and that the morphism is defined componentwise. Then we have the following theorem.

\begin{theorem}\label{product_theorem}
    For each $1 \leq i \leq m$, let $A_i/F$ and $f_i:A_i \rightarrow \bbp^{n_i}$ be as in Theorem~\ref{main_theorem}, and write $C_i := C(g_i,d_i,t_i)$. Let 
    \[A := A_1 \times \cdots \times A_m,\quad n := n_1+\cdots+n_m,\]
    and consider the product morphism 
    \[f:A \longrightarrow \bbp^{n_1} \times \cdots \times \bbp^{n_m},\quad f = (f_1,\ldots,f_m).\]
    Let $\Gamma$ be a subgroup of $A(F)$ of finite rank $r$. Then the constant 
    \[C := \prod_{i=1}^m C_i\]
    satisfies the following property.
    
    Fix affine charts $\bba^{n_i} \subseteq \bbp^{n_i}$ and identify 
    \[\bba^{n} = \bba^{n_1} \times \cdots \times \bba^{n_m} \subseteq \bbp^{n_1} \times \cdots \times \bbp^{n_m}.\]
    For every finite subset $X \subseteq f(\Gamma) \cap \bba^n$, 
    \[E(X) \leq C^{1+r}\lvert X \rvert^2,\qquad \lvert X+X \rvert \geq \left(C^{-1}\right)^{1+r}\lvert X \rvert^2.\]
\end{theorem}

The proof of Theorem~\ref{main_theorem} begins by fixing an affine chart $\bba^n \subseteq \bbp^n$ and considering the morphism
\[\Phi : V \times V \longrightarrow \bba^n,\quad (P,Q) \longmapsto f(P)+f(Q),\]
where $V=f^{-1}(\bba^n)$. For each $u \in \bba^n$, we define 
\[Y_u:=\overline{\Phi^{-1}(u)} \subseteq A^2.\]
The energy of a finite set $X \subseteq f(\Gamma) \cap \bba^n$ is then controlled by the number of pairs $(P,Q) \in \Gamma^2$ lying on these fibers $Y_u(F)$. For those $u \in \bba^n$ for which $Y_u$ does not contain any translate of a nontrivial abelian subvariety of $A^2$, uniform Mordell-Lang conjecture gives the required bound for $\lvert Y_u(F)\cap \Gamma^2 \rvert$. The problem is therefore reduced to understanding the exceptional set of $u \in \bba^n$ for which $Y_u$ contains a translate of a nontrivial abelian subvariety of $A^2$.

This is precisely where the simplicity of $A$ enters. In contrast with \cite{HMS26}, where the key step is to establish the non-degeneracy of $\mathcal V_{\mathrm{sum}}$, we analyze the exceptional set $\Sigma$ directly. Since $A$ is simple, every nontrivial proper abelian subvariety of $A^2$ is of the form
\[A(\alpha,\beta):=\{(\alpha P,\beta P)\:|\:P\in A\}.\]
Hence if $Y_u$ contains a translate of a nontrivial abelian subvariety of $A^2$, then we obtain a functional equation 
\[f(a+\alpha P)+f(b+\beta P)\equiv u.\]
We prove that this identity can occur only in finitely many cases, which yields the finiteness of $\Sigma$.

The paper is organized as follows. In Section~2, we prove preliminary combinatorial lemmas needed in the proof of Theorem~\ref{main_theorem} and Theorem~\ref{product_theorem}. In Section~3, we reduce Theorem~\ref{main_theorem} in geometric terms and recall the uniform Mordell-Lang conjecture needed in Section~4. In Section~4, we prove a uniform bound for $Y_u(F) \cap \Gamma^2$ for $u \notin \Sigma$ and in Section~5, we prove that the exceptional set $\Sigma$ is finite. In Section~6, we complete the proof of Theorem~\ref{main_theorem} using the bounds established in Section~4 and 5, and then prove Theorem~\ref{product_theorem}.

\section{Preliminary combinatorics}

Let $X$, $Y$ be finite subsets of a torsion-free abelian group $Z$. The sumset $X+Y$ is defined by 
\[X+Y := \{a+b\:|\:a \in X,\:b \in Y\},\]
and the energy $E(X,Y)$ between $X$ and $Y$ is defined by 
\[E(X,Y) := \lvert \{(a,b,c,d) \in (X \times Y)^2\:|\:a+b=c+d\} \rvert.\]
In particular, the energy of $X$ is defined by 
\[E(X) := E(X,X).\]

For each $u \in Z$, define 
\[N_u := \lvert \{(a,b) \in X \times Y\:|\:a+b=u\} \rvert.\]
Then we have 
\[\lvert X \rvert\lvert Y \rvert = \sum_{u \in X+Y} N_u,\quad E(X,Y) = \sum_{u \in X+Y} N_u^2.\]
By the Cauchy-Schwarz inequality, we obtain 
\begin{equation}\label{Cauchy_Schwarz_XY}
    \lvert X \rvert^2\lvert Y \rvert^2 \leq E(X,Y)\lvert X+Y \rvert.
\end{equation}
In particular, taking $X=Y$ in \eqref{Cauchy_Schwarz_XY} gives 
\begin{equation}\label{Cauchy_Schwarz_X}
    \lvert X \rvert^4 \leq E(X)\lvert X+X \rvert.
\end{equation}

We next observe that 
\begin{align*}
    E(X,Y) &= \lvert \{(a,b,c,d) \in (X \times Y)^2\:|\:a+b=c+d\} \rvert \\
    &= \lvert \{(a,c,b,d) \in X^2 \times Y^2\:|\:a-c=d-b\} \rvert \\
    &= \sum_{u \in Z} \lvert \{(a,c) \in X^2\:|\:a-c=u\} \rvert\lvert \{(b,d) \in Y^2\:|\:b-d=u\} \rvert.
\end{align*}
In particular, 
\[E(X) = \sum_{u \in Z} \lvert \{(a,c) \in X^2\:|\:a-c=u\} \rvert^2,\quad E(Y) = \sum_{u \in Z} \lvert \{(b,d) \in Y^2\:|\:b-d=u\} \rvert^2.\]
By the Cauchy-Schwarz inequality, 
\begin{equation}\label{Energy_Cauchy_Schwarz}
    E(X,Y)^2 \leq E(X)E(Y).
\end{equation}

We finally need an estimate of the energy $E(X)$ when the ambient group $Z$ is given by the product form.
\begin{lemma}\label{product_lemma}
    Let $X$ be a finite subset of a product group $Z \times W$, where $Z$ and $W$ are torsion-free abelian groups. For each $a \in Z$, define 
    \[X_a := \{x \in W\:|\:(a,x) \in X\}.\]
    Then 
    \[E(X) \leq \sum_{u \in Z}\left(\sum_{a+b=u} E(X_a,X_b)^{1/2}\right)^2.\]
\end{lemma}
\begin{proof}
    Note that for each $(u,v) \in Z \times W$, 
    \begin{align*}
        N_{(u,v)} &= \lvert \{((a,x),(b,y)) \in X^2\:|\:(a,x)+(b,y)=(u,v)\} \rvert \\
        &= \sum_{a+b=u} \lvert \{(x,y) \in X_a \times X_b\:|\:x+y=v\} \rvert.
    \end{align*}
    By the Cauchy-Schwarz inequality, 
    \[N_{(u,v)}^2 \leq \left(\sum_{a+b=u} E(X_a,X_b)^{1/2}\right)\left(\sum_{a+b=u} \frac{\lvert \{(x,y) \in X_a \times X_b\:|\:x+y=v\} \rvert^2}{E(X_a,X_b)^{1/2}}\right).\]
    Summing over $v \in W$ gives 
    \[\sum_{v \in W} N_{(u,v)}^2 \leq \left(\sum_{a+b=u} E(X_a,X_b)^{1/2}\right)^2.\]
    Summing over $u \in Z$ gives 
    \[E(X) = \sum_{u \in Z}\sum_{v \in W} N_{(u,v)}^2 \leq \sum_{u \in Z}\left(\sum_{a+b=u} E(X_a,X_b)^{1/2}\right)^2.\]
\end{proof}

\section{Reduction of Theorem~\ref{main_theorem}} 

In this section we reformulate Theorem~\ref{main_theorem} in geometric terms. The basic observation is that the additive energy $E(X)$ can be controlled by counting pairs of points in $\Gamma^2$ with the same image under $f(P)+f(Q)$. This leads naturally to the family of closed subschemes $Y_u$, whose geometry will govern the contribution of each term $N_u$ in the energy calculation.

Let $A/F$ be a simple abelian variety of dimension $g$ over an algebraically closed field $F$ of characteristic 0. Let $f:A \rightarrow \bbp^n$ be a morphism which is finite of degree $d$ onto its image $Z$, and let $t$ denote the projective degree of $Z$ in $\bbp^n$. Let $\Gamma$ be a subgroup of $A(F)$ of finite rank $r$. 

Fix an affine chart $\bba^n \subseteq \bbp^n$ and set $U = Z \cap \bba^n$, $V = f^{-1}(\bba^n) \subseteq A$. Then we have a morphism 
\[\Phi : V \times V \longrightarrow \bba^n,\quad (P,Q) \longmapsto f(P)+f(Q).\]
For each $u \in \bba^n$, define 
\[Y_u := \overline{\Phi^{-1}(u)} \subseteq A^2,\]
where the closure is the Zariski closure in $A^2$. Then $Y_u$ is a closed subscheme of $A^2$. 

Suppose a finite subset $X \subseteq f(\Gamma) \cap \bba^n$ is given. For each $u \in \bba^n$, define 
\[N_u := \lvert \{(a,b) \in X^2\:|\:a+b=u\} \rvert\]
and 
\[M_u := \lvert \{(P,Q) \in \Gamma^2\:|\:f(P),f(Q) \in \bba^n,\:f(P)+f(Q)=u\} \rvert.\]
Then we have a trivial estimate 
\begin{equation}\label{N_u_bound}
    N_u \leq \min\{M_u,\lvert X \rvert\}.
\end{equation}
We also note that the set 
\[\{(P,Q) \in \Gamma^2\:|\:f(P),f(Q) \in \bba^n,\:f(P)+f(Q)=u\}.\]
is contained in 
\[Y_u(F) \cap \Gamma^2,\]
hence 
\[M_u \leq \lvert Y_u(F) \cap \Gamma^2 \rvert.\]
This shows that the key issue is to bound $Y_u(F) \cap \Gamma^2$.

We will prove that there exists a finite set $\Sigma \subseteq \bba^n$ such that 
\begin{equation}\label{Sigma_finite}
    \lvert \Sigma \rvert \leq C_2(g,d,t)
\end{equation}
and for every $u \notin \Sigma$, 
\begin{equation}\label{Y_u_bounded}
    \lvert Y_u(F) \cap \Gamma^2 \rvert \leq C_1(g,d,t)^{1+r}.
\end{equation}
The first statement controls the number of exceptional $u \in \bba^n$, while the second gives a uniform bound for the remaining fibers. The proof of Theorem~\ref{main_theorem} directly follows from \eqref{Sigma_finite} and \eqref{Y_u_bounded}, as we can see in Section~6.

The main idea is that the exceptional set $\Sigma$ corresponds to those $u \in \bba^n$ for which $Y_u$ contains a translate of a nontrivial abelian subvariety, whereas for $u \notin \Sigma$ one expects a uniform control of $Y_u(F) \cap \Gamma^2$ from the uniform Mordell-Lang conjecture. We lastly recall the uniform Mordell-Lang conjecture proven by Gao-Ge-K\"{u}hne \cite{GGK21}.

\begin{theorem}[Uniform Mordell-Lang conjecture]\label{uniform_Mordell_Lang}
    Let $A/F$ be a polarized abelian variety of dimension $g$, $X$ be an irreducible closed subvariety of $A$ of degree $d$ (with respect to the polarization), and $\Gamma$ be a subgroup of $A(F)$ of finite rank $r$. If $X$ does not contain any translate of a nontrivial abelian subvariety of $A$, then $X(F) \cap \Gamma$ is finite and 
    \[\lvert X(F) \cap \Gamma \rvert \leq c(g,d)^{1+r}.\]
\end{theorem}
\begin{proof}
    This is a special case of \cite[Theorem~1.1]{GGK21}.
\end{proof}

\section{Uniform bound for $Y_u(F) \cap \Gamma^2$}

In this section we prove Proposition~\ref{C_1_bound}, which corresponds to the bound \eqref{Y_u_bounded}. The strategy is to compare each fiber $Y_u$ with the closed subscheme $D_u$ in $\bbp^n$. This allows us to bound both the degrees and the number of irreducible components of $Y_u$. Once these geometric bounds are established, Theorem~\ref{uniform_Mordell_Lang} yields the required estimate for $Y_u(F) \cap \Gamma^2$.

\begin{proposition}\label{C_1_bound}
    If $Y_u$ does not contain any translate of a nontrivial abelian subvariety of $A^2$, then 
    \[\lvert Y_u(F) \cap \Gamma^2 \rvert \leq C_1(g,d,t)^{1+r}.\]
\end{proposition}
\begin{proof}
    We will apply Theorem~\ref{uniform_Mordell_Lang} to an abelian variety $A^2/F$ and to each irreducible component $Y$ of $Y_u$. The polarization of $A^2$ is induced by $f$: 
    \[\mathcal L := (p_1^A)^*\mathcal{L}_0 \otimes (p_2^A)^*\mathcal{L}_0\]
    where 
    \[\mathcal{L}_0 := f^*\mathcal{O}_{\bbp^n}(1).\]

    For each $u \in \bba^n$, define 
    \[\mu_u:\bba^n \longrightarrow \bba^n,\quad x \longmapsto u-x\]
    and let $\bar \mu_u:\bbp^n \rightarrow \bbp^n$ be its projective extension. Then $\bar \mu_u$ is a projective automorphism of $\bbp^n$. 
    
    Recall that $Z=f(A)$ and $U = Z \cap \bba^n$. Define 
    \[C_u := Z \cap \bar \mu_u(Z)\subseteq \bbp^n,\quad D_u := \overline{U \cap \mu_u(U)} \subseteq \bbp^n.\]
    Then 
    \[C_u \cap \bba^n = U \cap \mu_u(U),\]
    so $D_u$ is exactly the union of the irreducible components of $C_u$ which meet $\bba^n$. 

    Since $\bar \mu_u$ is a projective automorphism of $\bbp^n$, 
    \[\deg(\bar \mu_u(Z))=\deg(Z)=t.\]
    By the generalized B\'{e}zout theorem, 
    \[\sum_{D \in \mathrm{Irr}(D_u)} \deg(D) \leq \sum_{C \in \mathrm{Irr}(C_u)} \deg(C) \leq t^2.\]
    In particular, each irreducible component $D$ of $D_u$ satisfies
    \[\deg(D) \leq t^2,\]
    and the number of irreducible components of $D_u$ is at most $t^2$.
    
    Next define 
    \[\Gamma_u:=\overline{\{(x,y)\in U^2 \mid x+y=u\}} \subseteq \bbp^n \times \bbp^n.\]
    The set under the closure is the graph of $\mu_u$ on $U \cap \mu_u(U)$. Therefore, $\Gamma_u$ is the graph of $\bar \mu_u$ on $D_u$:
    \[\Gamma_u = \{(x,\bar \mu_u(x))\:|\:x \in D_u\}.\]
    In particular, the first projection induces an isomorphism $p_1:\Gamma_u \cong D_u$. Moreover,
    \[(f\times f)^{-1}(\Gamma_u) \cap (V \times V) = \Phi^{-1}(u),\]
    so $Y_u$ is the union of the irreducible components of $(f\times f)^{-1}(\Gamma_u)$ which meet $V\times V$.
    
    Now let $Y$ be an irreducible component of $Y_u$. Then $Y$ is an irreducible component of $(f\times f)^{-1}(\Gamma_u)$, so its image under $f \times f$ is an irreducible component $\Delta$ of $\Gamma_u$, and this corresponds to an irreducible component $D$ of $D_u$ via $p_1:\Gamma_u \cong D_u$. 
    
    Consider the morphism
    \[\pi:= p_1 \circ (f \times f) = f \circ p_1^A : Y \longrightarrow D.\]
    Then the morphism $\pi$ is finite and $\deg(\pi) \leq d^2$. Indeed, for generic $x \in D \cap \bba^n$, the fiber $\pi^{-1}(x)$ is contained in
    \[f^{-1}(x)\times f^{-1}(u-x),\]
    hence has cardinality at most $d^2$.
    
    Recall that
    \[\mathcal{L}_0 := f^*\mathcal{O}_{\bbp^n}(1),\quad \mathcal L := (p_1^A)^*\mathcal{L}_0 \otimes (p_2^A)^*\mathcal{L}_0.\]
    On $\Delta$, the two projections $p_1$ and $p_2$ satisfy 
    \[p_2 = \bar\mu_u \circ p_1.\]
    Therefore, on $Y$, 
    \[f \circ p_2^A = p_2 \circ (f \times f) = \bar \mu_u \circ p_1 \circ (f \times f) = \bar \mu_u \circ f \circ p_1^A.\]
    Since $\bar\mu_u^*\mathcal O_{\bbp^n}(1)\simeq \mathcal O_{\bbp^n}(1)$, we obtain 
    \[\left.(p_1^A)^*\mathcal L_0\right|_Y \simeq \left.(p_2^A)^*\mathcal L_0\right|_Y \simeq \pi^*\mathcal O_D(1).\]
    Hence
    \[\left.\mathcal L\right|_Y=\left.(p_1^A)^*\mathcal L_0\right|_Y\otimes \left.(p_2^A)^*\mathcal L_0\right|_Y \simeq \pi^*\mathcal O_D(2).\]
    The dimension of $Y$ and $D$ is equal to $m \leq g$. Therefore
    \begin{align*}
        \deg_{\mathcal L}(Y) &= (\mathcal L|_Y)^m = (\pi^*\mathcal O_D(2))^m = \deg(\pi)\cdot \deg_{\mathcal O_{\bbp^n}(2)}(D) \\
        &= \deg(\pi)\cdot 2^m \deg(D) \leq 2^gd^2t^2.
    \end{align*}
    Thus every irreducible component of $Y_u$ has degree at most $2^gd^2t^2$.
    
    We also need a uniform bound for the number of irreducible components of $Y_u$. For each irreducible component $D$ of $D_u$, there are at most $d^2$ irreducible
    components of $Y_u$ over $D$ because $\deg(\pi) \leq d^2$. Therefore, the number of irreducible components of $Y_u$ is at most $d^2t^2$.

    At this point all geometric quantities appearing in Theorem~\ref{uniform_Mordell_Lang} are bounded independently of $u$. We now finish the proof. Suppose $Y_u$ does not contain any translate of a nontrivial abelian subvariety of $A^2$. Then the same is true for the irreducible component $Y$ of $Y_u$. By Theorem~\ref{uniform_Mordell_Lang} with the above degree calculation, we obtain 
    \[\lvert Y(F) \cap \Gamma^2 \rvert \leq c(2g,2^gd^2t^2)^{1+2r}.\]
    As $Y_u$ has at most $d^2t^2$ irreducible components, we conclude that 
    \[\lvert Y_u(F) \cap \Gamma^2 \rvert \leq d^2t^2 \cdot c(2g,2^gd^2t^2)^{1+2r} \leq C_1(g,d,t)^{1+r}.\]
\end{proof}

\section{Finiteness of $\Sigma$}

In this section, we prove Proposition~\ref{C_2_bound}, which corresponds to the bound \eqref{Sigma_finite}. We prove that $u \in \bba^n$ for which $Y_u$ contains a translate of a nontrivial abelian subvariety form a finite set. The proof heavily uses the simplicity of $A$ to obtain a rigid description of such subvarieties in $A^2$.

\begin{proposition}\label{C_2_bound}
    Define $\Sigma$ to be the set of $u \in \bba^n$ such that $Y_u$ contains a translate of a nontrivial abelian subvariety. Then $\Sigma$ is finite and 
    \[\lvert \Sigma \rvert \leq C_2(g,d,t).\]
\end{proposition}
\begin{proof}
    Since $A$ is simple, a nontrivial proper abelian subvariety of $A^2$ is given by the following form:
    \[A(\alpha,\beta) := \{(\alpha P,\beta P)\:|\:P \in A\},\quad \alpha,\beta \in \End(A),\quad (\alpha,\beta) \neq (0,0).\]
    
    We begin the proof by translating the condition $u \in \Sigma$ into a functional equation of the form 
    \[f(a+\alpha P) + f(b+\beta P) \equiv u,\]
    where $a,b \in A$ and $\alpha,\beta \in \End(A)$ with $\alpha,\beta$ both nonzero.
    
    Suppose $u \in \Sigma$. Then there exist $\alpha,\beta \in \End(A)$ and $a,b \in A$ such that $(\alpha,\beta) \neq (0,0)$ and 
    \[(a,b) + A(\alpha,\beta) \subseteq Y_u.\]
    We prove that $(a,b)+A(\alpha,\beta)$ must intersect $V \times V$.
    
    Assume $(a,b)+A(\alpha,\beta)$ is disjoint from $V \times V$. If $\alpha \neq 0$ and $\beta \neq 0$, then both are nonzero isogenies. Therefore, under the projections $p_1$ and $p_2$, $(a,b)+A(\alpha,\beta)$ map onto $A$. Since $V$ is an affine open set of $A$, there exists $P \in A$ such that $a+\alpha P,b+\beta P$ are both contained in $V$. This is a contradiction, so $\alpha=0$ or $\beta=0$. Without loss of generality, assume $\alpha=0$. Let $H$ be the hyperplane at infinity in $\bbp^n$. Recall the projective automorphism $\bar\mu_u$ of $\bbp^n$. Since $\mu_u:x \mapsto u-x$ on $\bba^n$, $\bar\mu_u$ is the identity on $H$. Hence if $(P,Q) \in Y_u$ and $(P,Q) \notin V \times V$, then $f(P)=f(Q)$. Since $(a,b)+A(0,\beta) \subseteq Y_u$ but $(a,b)+A(0,\beta)$ is disjoint from $V \times V$, we conclude that $f$ is constant, which is a contradiction.

    Since $(a,b)+A(\alpha,\beta)$ intersects with $V \times V$, we have 
    \[f(a+\alpha P) + f(b+\beta P) \equiv u\]
    as rational functions. If $\alpha=0$, then $f(b+\beta P)$ is constant and $\beta \neq 0$, hence $f$ is constant, which is a contradiction. Thus $\alpha \neq 0$ and by symmetry, $\beta \neq 0$. It follows that $\alpha$ and $\beta$ are both nonzero isogenies.

    We now remove the translation symmetries of $f$ by passing to the quotient by its stabilizer. Let 
    \[G := \{t \in A \:|\: f \circ \tau_t = f\}\]
    be the translation stabilizer of $f$. Since $f$ is finite, $G$ is finite. Let $B=A/G$, $q:A \rightarrow B$ be the quotient morphism, and let $f = g \circ q$. Then $g:B \rightarrow \bbp^n$ has a trivial translation stabilizer.
    
    Passing to the quotient $q:A \rightarrow B$ gives 
    \[g(\bar a+\bar \alpha P)+g(\bar b+\bar \beta P)\equiv u,\]
    where $\bar \alpha,\bar \beta:A \rightarrow B$ are isogenies. If $T \in \ker \bar \alpha$, then comparing the functional equation at $P$ and $P+T$ yields
    \[g(\bar b+\bar \beta(P+T)) \equiv g(\bar b+\bar\beta P).\]
    Since $g$ has a trivial translation stabilizer, this implies $\bar \beta(T)=0$. Hence $\ker \bar \alpha \subseteq \ker \bar \beta$, and by symmetry, $\ker \bar \alpha=\ker \bar \beta$. Therefore there exists $\gamma \in \Aut(B)$ such that
    \[\bar \beta=\gamma \bar \alpha.\]
    
    Set $Q:=\bar a+\bar \alpha P$. Then $Q$ runs through $B$, and the functional equation becomes
    \[g(Q)+g(c+\gamma Q)\equiv u\]
    for some $c \in B$.
    
    Let $\mathcal{L}_B=g^*\mathcal O_{\bbp^n}(1)$ be the polarization of $B$. Recall that we have chosen affine coordinate $\bba^n$ in $\bbp^n$. Let $H$ be the hyperplane at infinity and define 
    \[D := g^*H \sim \mathcal{L}_B.\]
    
    From
    \[g(P)+g(c+\gamma P)\equiv u,\]
    by considering the pole divisors of affine coordinate of $g$, we obtain 
    \[(\tau_c\circ[\gamma])^*D=D.\]
    Passing to divisor classes in the Neron-Severi group $\mathrm{NS}(B)$, we obtain
    \[\gamma^*[D]=[D].\]
    Since $D \sim \mathcal{L}_B$, this yields
    \[\gamma^*[\mathcal{L}_B]=[\mathcal{L}_B].\]
    Therefore 
    \[\gamma \in \Aut(B,[\mathcal{L}_B]).\]

    It remains to bound the number of possible translations $c$ for a fixed automorphism $\gamma \in \Aut(B,[\mathcal{L}_B])$. Fix such a $\gamma$. If both $c$ and $d$ satisfy
    \[g(P)+g(c+\gamma P)\equiv u,\qquad g(P)+g(d+\gamma P)\equiv v,\]
    then we obtain 
    \[(\tau_c\circ[\gamma])^*D=(\tau_{d}\circ[\gamma])^*D=D.\]
    Hence
    \[\tau_{c-d}^*D=D,\]
    so
    \[c-d \in \{t\in B\mid \tau_t^*D=D\} \subseteq \{t\in B\mid \tau_t^*\mathcal{L}_B\simeq \mathcal{L}_B\} =: K(\mathcal{L}_B).\]
    Since $\mathcal{L}_B$ is ample, $K(\mathcal{L}_B)$ is finite. Therefore, for fixed $\gamma$, the number of $c$ that can occur is at most $\lvert K(\mathcal{L}_B) \rvert$.
    
    We have proven that for fixed $\gamma$, the number of $u \in \Sigma$ that can occur is at most $\lvert K(\mathcal{L}_B) \rvert$. By \cite[Proposition~17.5]{Mil86}, $\Aut(B,[\mathcal{L}_B])$ is finite, so the number of $\gamma$ that can occur is $\lvert \Aut(B,[\mathcal{L}_B]) \rvert$. We conclude that the set $\Sigma$ is finite and 
    \[\lvert \Sigma \rvert \leq \lvert \Aut(B,[\mathcal{L}_B]) \rvert\lvert K(\mathcal{L}_B) \rvert.\]
    
    Now the next two lemmas end the proof of the proposition. 
\end{proof}

\begin{lemma}
    We have 
    \[\lvert \Aut(B,[\mathcal L_B]) \rvert \leq \lvert \mathrm{GL}_{2g}(\bbf_3) \rvert \leq 3^{4g^2}.\]
\end{lemma}
\begin{proof}
    Since $F$ is an algebraically closed field of characteristic 0, we have
    \[B[3] := B[3](F) \cong (\bbz/3\bbz)^{2g}.\]
    By \cite[Proposition~17.5]{Mil86}, any automorphism in $\Aut(B,[\mathcal L_B])$ acting trivially on $B[3]$ is the identity. Therefore the natural action of $\Aut(B,[\mathcal L_B])$ on $B[3]$ is faithful, and hence
    \[\Aut(B,[\mathcal L_B]) \hookrightarrow \Aut(B[3]) \cong \mathrm{GL}_{2g}(\bbf_3).\]
\end{proof}

\begin{lemma}
    We have 
    \[\lvert K(\mathcal L_B) \rvert \leq d^2t^2.\]
\end{lemma}
\begin{proof}
    Fix a polarization isogeny $\lambda:B \rightarrow \hat{B}$ associated to $\mathcal L_B$. Recall that 
    \[K(\mathcal L_B) = \ker(\lambda).\]
    Since $\lambda$ is an isogeny associated to $\mathcal L_B$, by \cite[Theorem~13.3]{Mil86}
    \[\lvert K(\mathcal L_B) \rvert=\deg(\lambda)=\left(\frac{\mathcal L_B^g}{g!}\right)^2.\]
    
    Recall that $f=g \circ q$ and $q:A \rightarrow B$ is the quotient with kernel $G$. Hence
    \[q^*\mathcal L_B=f^*\mathcal O_{\bbp^n}(1)=\mathcal L_0.\]
    Since $q$ is finite of degree $\lvert G \rvert$, we obtain
    \[\lvert G \rvert\mathcal L_B^g=(q^*\mathcal L_B)^g=\mathcal L_0^g.\]
    Since $f:A \rightarrow f(A)$ has degree $d$ and $f(A)$ has projective degree $t$ in $\bbp^n$, 
    \[\mathcal L_0^g = dt.\]
    Therefore, 
    \[\lvert K(\mathcal L_B) \rvert =\left(\frac{\mathcal L_B^g}{g!}\right)^2 = \left(\frac{dt}{g!\lvert G \rvert}\right)^2 \leq d^2t^2.\]
\end{proof}

\section{Proof of Theorem~\ref{main_theorem} and Theorem~\ref{product_theorem}}

In this section, we complete the proof of Theorem~\ref{main_theorem}. We then introduce weighted generalization of Theorem~\ref{main_theorem}, and finally prove Theorem~\ref{product_theorem} by induction.

We first prove Theorem~\ref{main_theorem}. The theorem directly follows from the results of the last three sections.

\begin{theorem}\label{main_theorem_re}
    Let $A/F$ be a simple abelian variety of dimension $g$ over an algebraically closed field $F$ of characteristic 0. Let $f:A \rightarrow \bbp^n$ be a morphism which is finite of degree $d$ onto its image, and let $t$ denote the projective degree of $f(A)$ in $\bbp^n$. Let $\Gamma$ be a subgroup of $A(F)$ of finite rank $r$. Then there exists a constant $C(g,d,t)>0$ with the following property.
    
    For every affine chart $\bba^n \subseteq \bbp^n$ and every finite subset $X \subseteq f(\Gamma) \cap \bba^n$, 
    \[E(X) \leq C(g,d,t)^{1+r}\lvert X \rvert^2,\qquad \lvert X+X \rvert \geq \left(C(g,d,t))^{-1}\right)^{1+r}\lvert X \rvert^2.\]
\end{theorem}
\begin{proof}
    The bound \eqref{N_u_bound} implies 
    \[E(X) = \sum_{u \in X+X} N_u^2 = \sum_{\substack{u \in X+X \\ u \notin \Sigma}} N_u^2 + \sum_{\substack{u \in X+X \\ u \in \Sigma}} N_u^2 \leq \sum_{\substack{u \in X+X \\ u \notin \Sigma}} M_uN_u + \sum_{\substack{u \in X+X \\ u \in \Sigma}} \lvert X \rvert^2.\]
    By \eqref{Y_u_bounded}, the first sum is bounded by 
    \[C_1(g,d,t)^{1+r} \sum_{\substack{u \in X+X \\ u \notin \Sigma}} N_u \leq C_1(g,d,t)^{1+r} \lvert X \rvert^2\]
    and by \eqref{Sigma_finite}, the second sum is bounded by 
    \[C_2(g,d,t)\lvert X \rvert^2.\]
    Hence, the whole energy is bounded by 
    \[C(g,d,t)^{1+r}\lvert X \rvert^2.\]
    The lower bound for $\lvert X+X \rvert$ follows from \eqref{Cauchy_Schwarz_X}.
\end{proof}

We next introduce the weighted generalization of Theorem~\ref{main_theorem}. This generalization is needed in the proof of Theorem~\ref{product_theorem}.

\begin{theorem}\label{weighted_theorem}
    Let $A/F$ and $f:A \rightarrow \bbp^n$ be as in Theorem~\ref{main_theorem_re}, and let $\Gamma$ be a subgroup of $A(F)$ of finite rank $r$. Then for every affine chart $\bba^n \subseteq \bbp^n$, every finite subset $X \subseteq f(\Gamma) \cap \bba^n$, and every weight $w:X \rightarrow [0,\infty)$, 
    \[\sum_{u \in X+X} \left( \sum_{\substack{a+b=u \\ a,b \in X}} w(a)w(b) \right)^2 \leq C(g,d,t)^{1+r}\left(\sum_{a \in X} w(a)^2\right)^2.\]
\end{theorem}
\begin{proof}
    For each $u \in X+X$, define 
    \[N(w)_u := \sum_{\substack{a+b=u \\ a,b \in X}} w(a)w(b).\]
    We have to estimate the sum 
    \[\sum_{u \in X+X} N(w)_u^2 = \sum_{\substack{u \in X+X \\ u \notin \Sigma}} N(w)_u^2 + \sum_{\substack{u \in X+X \\ u \in \Sigma}} N(w)_u^2\]
    
    We first estimate the first sum. By the Cauchy-Schwarz inequality, 
    \[N(w)_u^2 \leq N_u\sum_{\substack{a+b=u \\ a,b \in X}} w(a)^2w(b)^2.\]
    Thus 
    \[\sum_{\substack{u \in X+X \\ u \notin \Sigma}} N(w)_u^2 \leq \sum_{\substack{u \in X+X \\ u \notin \Sigma}} N_u\sum_{\substack{a+b=u \\ a,b \in X}} w(a)^2w(b)^2.\]
    By \eqref{N_u_bound} and \eqref{Y_u_bounded}, this sum is bounded by 
    \[C_1(g,d,t)^{1+r} \sum_{\substack{u \in X+X \\ u \notin \Sigma}}\sum_{\substack{a+b=u \\ a,b \in X}} w(a)^2w(b)^2.\]
    However, 
    \[\sum_{\substack{u \in X+X \\ u \notin \Sigma}}\sum_{\substack{a+b=u \\ a,b \in X}} w(a)^2w(b)^2 \leq \sum_{u \in X+X}\sum_{\substack{a+b=u \\ a,b \in X}} w(a)^2w(b)^2 = \left(\sum_{a \in X} w(a)^2\right)^2.\]
    Hence, the first sum is bounded by 
    \[C_1(g,d,t)^{1+r}\left(\sum_{a \in X} w(a)^2\right)^2.\]

    We next estimate the second sum. By the Cauchy-Schwarz inequality, we obtain 
    \[N(w)_u^2 \leq \left(\sum_{\substack{a \in X \\ u-a \in X}} w(a)w(u-a)\right)^2 \leq \left(\sum_{a \in X} w(a)^2\right)^2.\]
    By \eqref{Sigma_finite}, the second sum is bounded by 
    \[C_2(g,d,t)\left(\sum_{a \in X} w(a)^2\right)^2.\]

    Hence, the whole sum is bounded by 
    \[C(g,d,t)^{1+r}\left(\sum_{a \in X} w(a)^2\right)^2.\]
\end{proof}

We are now ready to prove Theorem~\ref{product_theorem}.

\begin{theorem}\label{product_theorem_re}
    For each $1 \leq i \leq m$, let $A_i/F$ and $f_i:A_i \rightarrow \bbp^{n_i}$ be as in Theorem~\ref{main_theorem_re}, and write $C_i := C(g_i,d_i,t_i)$. Let 
    \[A := A_1 \times \cdots \times A_m,\quad n := n_1+\cdots+n_m,\]
    and consider the product morphism 
    \[f:A \longrightarrow \bbp^{n_1} \times \cdots \times \bbp^{n_m},\quad f = (f_1,\ldots,f_m).\]
    Let $\Gamma$ be a subgroup of $A(F)$ of finite rank $r$. Then the constant 
    \[C := \prod_{i=1}^m C_i\]
    satisfies the following property.
    
    Fix affine charts $\bba^{n_i} \subseteq \bbp^{n_i}$ and identify 
    \[\bba^{n} = \bba^{n_1} \times \cdots \times \bba^{n_m} \subseteq \bbp^{n_1} \times \cdots \times \bbp^{n_m}.\]
    For every finite subset $X \subseteq f(\Gamma) \cap \bba^n$, 
    \[E(X) \leq C^{1+r}\lvert X \rvert^2,\qquad \lvert X+X \rvert \geq \left(C^{-1}\right)^{1+r}\lvert X \rvert^2.\]
\end{theorem}
\begin{proof}
    We use induction on $m$. The base case $m=1$ is exactly Theorem~\ref{main_theorem_re}. 
    
    Suppose the theorem is proved for $m-1$. Let $p_1^A:A \rightarrow A_1$ be the first projection and $p_2^A:A \rightarrow A_2 \times \cdots \times A_m$ be the remaining projection. Let $\Gamma_1 = p_1^A(\Gamma)$ and $\Gamma' = p_2^A(\Gamma)$. Since $\Gamma$ has finite rank $r$, $\Gamma_1$ and $\Gamma'$ both have finite rank at most $r$. In particular, we are free to use induction hypothesis for $A_2 \times \cdots \times A_m$ and Theorem~\ref{weighted_theorem} for $A_1$.

    For each $a \in \bba^{n_1}$, define 
    \[X_a := \{x \in \bba^{n_2} \times \cdots \times \bba^{n_m}\:|\:(a,x) \in X\}.\]
    By Lemma~\ref{product_lemma}, 
    \[E(X) \leq \sum_{u \in \bba^{n_1}}\left(\sum_{a+b=u} E(X_a,X_b)^{1/2}\right)^2.\]
    By \eqref{Energy_Cauchy_Schwarz} and the induction hypothesis, 
    \[E(X_a,X_b) \leq E(X_a)^{1/2}E(X_b)^{1/2} \leq (C')^{1+r}\lvert X_a \rvert\lvert X_b \rvert\]
    where 
    \[C' = \prod_{i=2}^m C_i.\]
    Hence, we obtain 
    \[E(X) \leq (C')^{1+r}\sum_{u \in \bba^{n_1}}\left(\sum_{a+b=u} \lvert X_a \rvert^{1/2}\lvert X_b \rvert^{1/2}\right)^2.\]
    Now let $X_1$ be the image of $X$ under the first projection. By applying Theorem~\ref{weighted_theorem} for $X_1$ and the weight $w(a) = \lvert X_a \rvert^{1/2}$, we conclude that 
    \[E(X) \leq (C_1)^{1+r}(C')^{1+r}\left(\sum_{a \in X_1} \lvert X_a \rvert\right)^2 = C^{1+r}\lvert X \rvert^2.\]
    The lower bound for $\lvert X+X \rvert$ follows from \eqref{Cauchy_Schwarz_X}.
\end{proof}

\end{document}